\documentclass[a4paper, 12pt]{article}

\usepackage{amsfonts}
\usepackage {amssymb}
\usepackage {amsmath}
\usepackage {amsmath}
\usepackage {amsthm}
\usepackage{graphicx}
\usepackage{multirow}
\usepackage{xypic}
\usepackage {amscd}
\usepackage{mathrsfs}
\usepackage[colorlinks, linkcolor=blue, anchorcolor=black, citecolor=red]{hyperref}
\usepackage{enumerate}

\usepackage{geometry}  

\newcommand{\bC}{\mathbb{C}}

\newcommand{\bR}{\mathbb{R}}

\newcommand{\mcH}{\mathcal{H}}

\newcommand{\cF}{\mathcal{F}}
\newcommand{\cE}{\mathcal{E}}

\newcommand{\inv}{\mathrm{inv}}

\newtheorem{thm}{Theorem}[section]
\newtheorem{prop}[thm]{Proposition}
\newtheorem{defn}[thm]{Definition}

\newtheorem{cor}[thm]{Corollary}
\newtheorem{rem}[thm]{Remark}

\newtheorem{lem}[thm]{Lemma}

\newtheorem{defn-prop}[thm]{Definition-Proposition}

\begin{document}

\title{K\"{a}hler structures for holomorphic submersions}
\author{Chi Li}
\date{}

\maketitle

\abstract{
In this short paper, for any holomorphic submersion $\pi: X\rightarrow B$, we derive a criterion for $X$ to have K\"{a}hler structures. This criterion generalizes Blanchard's criterion for a special class of isotrivial holomorphic submersions. We use this criterion to answer a question of Harvey-Lawson in the case of fiber dimension one. As the main application, we prove that the existence of Hermitian-Symplectic structures on certain class of holomorphic submersions with K\"{a}hler fibers and K\"{a}hler bases implies that the total spaces are K\"{a}hler.  This class includes isotrivial submersions and torus fibrations. 
  }

\tableofcontents

\section{Introduction and main results}

Let $X, B$ be compact complex manifolds. Assume that there is a holomorphic submersion $\pi: X\rightarrow B$. 
We want to derive a criterion for $X$ to be a K\"{a}hler manifold.
 
If $X$ admits a K\"{a}hler metric with the associated K\"{a}hler form denoted by $Q$. Then $Q$ is a closed $(1,1)$-form which defines the K\"{a}hler class $[Q]\in H^2(X,\bR)$. 
In this paper, we will work with K\"{a}hler forms and do not distinguish them from the corresponding
K\"{a}hler metrics. 
Let $f_b: X_b\rightarrow X$ denote the embedding of the fiber $X_b=\pi^{-1}(b)$. 
The restriction $Q|_{X_b}=f_b^*Q$ is a K\"{a}hler metric on $X_b$ for any $b\in B$. On the cohomology level, we have natural restriction homomorphisms for $k\in \{0, \dots, \dim X\}$:
\begin{equation}\label{eq-fb*}
f_b^*: H^k(X, \bR)\rightarrow H^k(X_b, \bR).
\end{equation}
So we have the first necessary condition:

\vskip 1mm
\noindent
$\mathbf{(Condition\; I)}$ There is a class $[Q]\in H^{2}(X, \bR)$ that restricts to be a K\"{a}hler class on each fiber $X_b=\pi^{-1}(b)$, $b\in B$. 
\vskip 1mm

Moreover one can show that if $Q$ is a K\"{a}hler form on $X$, then a K\"{a}hler form on $B$ can be constructed as $\int_{X/B}Q^{n+1}$ where $\int_{X/B}$ denotes the operation of integration along fibers. So we get the second necessary condition:

\vskip 1mm
\noindent
$\mathbf{(Condition\; II)}$ $B$ is a K\"{a}hler manifold.

\vskip 1mm

The main result of this paper is that these two conditions are also sufficient:
\begin{thm}\label{thm-main1}
In the above setting of holomorphic submersions, there is a K\"{a}hler metric on $X$ if and only if $\mathbf{(Condition\; I)}$ and $\mathbf{(Condition\; II)}$ are both satisfied. 
\end{thm}

Note that if we know a priori that $[Q]\in H^2(X, \bR)$ is represented by a closed $(1,1)$-form, then the result can be proved quite easily by applying a family version of $\partial\bar{\partial}$-lemma. See Remark \ref{rem-11case}. However here we do not assume this condition. Indeed, the main goal is to construct such a closed $(1,1)$-form. In this aspect, one could think Theorem \ref{thm-main1} as a K\"{a}hler correspondent to Thurston's well-known construction of symplectic structures for symplectic fibrations (\cite{Thu76}, \cite[Theorem 4.91]{FOT08}). 

Let $X_b=\pi^{-1}(b)$ (with $b\in B$) be any fibre of $\pi$. There is a natural monodromy action of $\pi_1(B)$ on $H^k(X_b, \bR)$ defined by translating cohomology classes. If we denote by $H^k(X_b, \bR)^{\inv}$ the set of elements in $H^k(X_b, \bR)$ that are invariant under the monodromy action, then we know that the image of the restriction morphism $f_b^*$ in \eqref{eq-fb*} is contained in $H^k(X_b, \bR)^{\inv}$. In the simplest situation when the fiber has complex dimension 1, we have $H^{2}(X_b, \bR)^{\inv}=H^{1,1}(X_b, \bR)$ and Theorem \ref{thm-main1} can answer positively a question of Harvey-Lawson (see Theorem \ref{cor-HL}).

On the other hand, we have the Leray spectral sequence $\{E^{p,q}_r, d_r: E^{p,q}_r\rightarrow E^{p+r,q-r+1}_r\}$ that converges to $\{ Gr_{\cF_L}^p H^{p+q}(X, \bR)\} $ with respect to the Leray filtration $\cF_L$. 
There are isomorphisms:
\begin{equation}\label{eq-E2inv}
E^{0,k}_2=H^0(B, R^k\pi_*\bR)\cong H^k(X_b, \bR)^{\inv}.
\end{equation}
According to Leray spectral sequence, a class $x\in E^{0,k}_2$ extends to be a class in $H^k(X, \bR)$ if and only if $d_r x=0$ for $r\in \{2,\dots, k+1\}$ (see Section \ref{sec-Leray} for more explanations).
With these notations, there is a more refined version of Theorem \ref{thm-main1}. 
\begin{thm}\label{thm-main2}
There is a K\"{a}hler metric on $X$ if and only if the following conditions are all satisfied:

\noindent
$\mathbf{(Condition\; Ia)}$
There is an element $[\omega]\in H^0(B, R^2\pi_*\bR)$ restricts to be a K\"{a}hler class on $X_{b}$ for any $b\in B$. 

\noindent
$\mathbf{(Condition\; Ib)}$
$d_2 [\omega]=0$ in $E^{2,1}_2=H^2(B, R^1\pi_*\bR)$. 

\noindent
$\mathbf{(Condition\; II)}$
$B$ is K\"{a}hler. 
\end{thm}
The conditions of Theorem \ref{thm-main2} is a priori weaker than the conditions of Theorem \ref{thm-main1}. In other words, Theorem \ref{thm-main2} implies \ref{thm-main1} (the ``if" direction). 

Theorem \ref{thm-main2} generalizes an old result of Blanchard in \cite{Bla56} who considered a special class of isotrivial holomorphic submersions. Here a holomorphic submersion is called \textit{isotrivial} if all fibers are isomorphic as complex manifolds.
Note that by a Theorem of Fischer-Grauert \cite{FG65}, isotrivial holomorphic submersions are always locally trivial and are sometimes called analytic fiber bundles.  In fact, in proving the criterion in \cite{Bla56} Blanchard not only assumed isotriviality but also that $\pi_1(B)$ acts trivially $H^1(X_b)$.  

$\mathbf{(Conditions\; I)}$ appeared in \cite[Remark 4.16]{Voi02}, where it is pointed out that Deligne's theorem on the degeneration of Leray spectral sequence at the $E_2$-page for projective submersions can be proved under this assumption.
This condition also appeared in \cite{Nak99} where it is called \textit{cohomologically K\"{a}hler}. In fact,  in the special case of torus fibrations Theorem \ref{thm-main1} and Theorem \ref{thm-main2} are contained in \cite[Proposition 2.17]{Nak99} which was proved there by using a precise description of torus fibrations in terms of variation of Hodge structures of weight 1. 

Our motivation for proving Theorem \ref{thm-main1} comes from a question of Li-Zhang \cite{LZ09} and Streets-Tian \cite{ST10} about the existence of Hermitian-Symplectic (HS) structures on non-K\"{a}hler complex manifolds. By definition, a Hermitian-Symplectic structure on a complex manifold $X$ with the integrable almost complex structure $J$ is a symplectic structure $\omega$ on $X$
that tames $J$ which means that $\omega(v, Jv)>0$ for any $v\neq 0\in T_\bR X$. Any K\"{a}hler form $\omega$ on $(X, J)$ is Hermitian-Symplectic because in this case $\omega$ not just tames $J$ but is also compatible with $J$ which means that $\omega(Jv, Jw)=\omega(v, w)$ for any $v,w\in T_\bR X$. Li-Zhang and Streets-Tian asked whether there are examples of HS structure on non-K\"{a}hler complex manifolds. \footnote{In real dimension 4, Donaldson had a similar question about tamed/compatible symplectic structures but without assuming the integrability of complex structures.} The question is still open in general, though there are negative results which say that there are no such examples among complex surfaces (\cite{LZ09, ST10}), nilmanifolds with invariant complex structures (\cite{EFV12}), twistor spaces (\cite{Ver14}), Moishezon manifolds and complex manifolds of Fujiki class (\cite{Pet86, Chi14}). See also \cite{DP20, Ye22} for some analytic approach to the general problem. 
We consider this question for complex manifolds that admit structures of holomorphic submersions. As in the K\"{a}hler case, it is easy to see that the existence of HS structures for holomorphic submersions forces the fibers and the base to be HS. So if we are looking for such examples of lowest dimension, it makes sense to assume that the base and fibers are all K\"{a}hler. 
Even though such an example is currently missing to us, we will show that HS submersions always satisfy $\mathbf{(Condition\; Ib)}$ (Proposition \ref{prop-HSd2}). By applying criterion in Theorem \ref{thm-main2}, we can then derive a K\"{a}hlerian criterion for HS submersions  (Proposition \ref{prop-HSKahler}) and rule out several classes of (non-trivial) examples:
\begin{thm}\label{thm-HSfiber}
Let $\pi: X\rightarrow B$ be a holomorphic submersion with K\"{a}hler fibers and a K\"{a}hler base. 
Assume that there is a Hermitian-Symplectic structure on $X$. Let $F$ denote a fiber of $\pi$. 
Then $X$ must be K\"{a}hler if one of the following conditions is satisfied:
\begin{enumerate}
\item The holomorphic submersion is isotrivial. 
\item The fibers of $\pi$ are complex tori (of possibly varying complex structures). 
\item The monodromy action of $\pi_1(B)$ on $H^2(F)$ is trivial. 
\item $H^{2,0}(F)=0$. 
\end{enumerate}
\end{thm}

We end this introduction by sketching the organization of this note. In section \ref{sec-Leray}, we will review the Leray spectral sequence in terms of filtered de Rham complex, describe $d_1: E^{p,q}_1\rightarrow E^{p+1,q}_1$ in terms of Gauss-Manin connection and explain its compatibility with the exterior differential operator. In section \ref{sec-Deligne}, we will review Deligne's Hodge theory for the cohomology groups with coefficients of polarized variation of Hodge structures and derive a variant of the classical $\partial\bar{\partial}$-lemma that will be crucial for our construction. In section \ref{sec-main}, we will construct a K\"{a}hler form step-by-step which is guided by the Leray spectral sequence as more explained in more detail in Section \ref{sec-Leray}. The main line of construction is similar to the original construction of Blanchard but differs from it significantly in some steps because we are dealing with a more general situation (see Remark \ref{rem-Blacons}). Finally in section \ref{sec-HS}, we will apply the criterion in Theorem \ref{thm-main2} to study Hermitian-Symplectic structures for holomorphic submersions.

\textbf{Acknowledgement:} The author is grateful to Weiyi Zhang for a careful reading of the draft of this paper and providing valuable feedback and corrections. He also thanks Jeffrey Streets and Gang Tian for very helpful comments and notifying him the papers \cite{DP20, Ye22}.

\section{Preliminaries}

\subsection{Leray spectral sequences}\label{sec-Leray}

Let $\pi: X\rightarrow B$ be a differentiable fiber bundle. In this subsection, we do not consider any complex structures. It is well-known that there is a Leray spectral sequence: $\{E^{p,q}_r, d_r\}$ with 
\begin{equation*}
E^{p,q}_2=H^p(B, R^q\pi_* \bR)\Rightarrow E^{p,q}_\infty \cong Gr H^{p+q}(X, \bR). 
\end{equation*}

To prepare for the later construction, we explain this convergence of spectral sequence from the point view of filtered de Rham complex in terms of differential forms. 
In the following discussion, for any local calculations, we choose a small coordinate neighborhood $\{U, u=(u_1,\dots, u_m)\}$ of $B$ such that there is a diffeomorphism 
\begin{equation}\label{eq-lctriv}
\pi^{-1}(U)\cong U\times F.
\end{equation} 
Choose local real coordinates $y=\{y_1,\dots, y_n\}$ on $F$ such that $\pi: \pi^{-1}(U) \rightarrow U$ is given by the projection $\pi(u,y)=u$. 

Denote by $\cE^{\ell}(X)$ the space of smooth degree-$\ell$ differential forms on $X$. 
A form $\eta\in \cE^{\ell}(X)$ is locally represented by the following expression:
\begin{eqnarray}\label{eq-etalc}
\eta&=&\sum_{|I|+|K|=\ell} \eta_{I;K}(u, y) du_{I}\wedge d y_{K} =\sum_{I} du_I\wedge \eta_I
\end{eqnarray}
where, for the simplicity of notation, for an index set $I=\{i_1,\dots, i_{|I|}\}$, we denote $du_I=du_{i_1}\wedge \cdots\wedge du_{|I|}$ (similarly for $dy_K$), and $\eta_I=\sum_{|K|=\ell-|I|}\eta_{I;K}dy_K$. 

Consider the following Leray filtration:
\begin{equation*}
\cF_L^{p}\cE^\ell(X)=\left\{ \eta=\sum_{|I|+|K|=\ell} \eta_{I;K} d u_I\wedge dy_K; |I|\ge p \right\}.
\end{equation*}
It is easy to verify that under the change of coordinates $(u, y')=(u, y'(u, y))$, for any $I$ with $|I|=p$ the term $\sum_K \eta_{I;K} dy_K$ changes as tensor along any fibre $X_b=\pi^{-1}(b)$. 
In particular, the restriction $\eta_I|_{X_b}:=f_b^*\eta_I$ is a well-defined $(\ell-p)$-form on any fiber $X_b$. 

Let's recall the description of the Leray spectral sequence following \cite[p. 440-442]{GH78}:
\begin{equation}\label{eq-Erpq}
E^{p,q}_r=\frac{\{a\in \cF_L^p \cE^{p+q}(X): da\in \cF_L^{p+r}\cE^{p+q+1}(X) \}}{d(\cF_L^{p-r+1}\cE^{p+q-1}(X))+\cF_L^{p+1}\cE^{p+q}(X)}.
\end{equation}
Here the denominator is not a subgroup of the numerator and the meaning is that we take $\{\text{denominator as written}\}\cap \{\text{numerator} \}$. 
The differential $d_r: E^{p,q}_r\rightarrow E^{p+r,q-r+1}_r$ is then defined as:
\begin{equation}\label{eq-dra}
d_r [a]=[da] \in E^{p+r,q-r+1}_r=\frac{\{b\in \cF_L^{p+r}\cE^{p+q+1}(X); db\in \cF_L^{p+2r}\cE^{p+q+2}(X) \}}{d \cF_L^{p+1} \cE^{p+q}(X)+\cF_L^{p+r+1}\cE^{p+q+1}(X)}.
\end{equation} 
If a class $[\eta]\in E^{p,q}_r$
satisfies $d_r [\eta]=0$, then by \eqref{eq-dra} there exists $\chi\in \cF_L^{p+1}\cE^{p+q}$ such that $d \eta-d \chi\in \cF_L^{p+r+1}\cE^{p+q+1}$ and hence $\eta$ and $\eta'=\eta-\chi$ defines the same class in $E^{p,q}_{r}$ while $\eta'=\eta-\chi$ defines a class in $E^{p,q}_{r+1}$.

When $p=0$, we have the map
$d_r: E^{0,q}_r\rightarrow E^{r, q-r+1}_r$ and $E^{0,q}_{r+1}=\mathrm{Ker}(d_r)\subseteq E^{0,q}_r$. So we always have the inclusion $E^{0,q}_r\subseteq E^{0,q}_{r-1}\subseteq \cdots \subseteq E^{0,q}_1$. 

Consider the 2nd page (i.e. $r=2$) and denote by $V^q=R^q\pi_*\bR$ the local system with fiber $\{H^q(X_b, \bR); b\in B\}$. Then there is a natural identification:
\begin{equation}\label{eq-invsp}
E^{0,q}_2=H^0(B, V^q)=H^q(X_b,\bR)^{\inv}
\end{equation} 
where $H^0(B, V^q)$ is the space of flat sections of $V^q$ and $H^q(X_b, \bR)^{\inv}$  is the space of cohomology that is invariant under the monodromy action.  Moreover, $[\eta]\in E^{0,q}_r\subseteq E^{0,q}_2=H^0(B, V^q)$ if and only if we can find a representative $\eta$ of $[\eta]\in E^{0,q}_2$ such that $d\eta\in \cF^r \cE^{q+1}(X)$. In this case, $d_r [\eta]=0$ if and only if 
there exists $\chi\in \cF_L^1\cE^q$ such that $d(\eta-\chi)\in \cF_L^{r+1} \cE^{q+1}(X)$. 
As a consequence, if $d_r [\eta]=0$ for any $r\in \{1, \dots, q+1\}$, then the representative $\eta$ can be chosen to be a closed $q$-form and we say that $[\eta]\in H^0(B, V^q)$ extends to be a closed $q$-form on $X$. 
In other words, the subspace $E^{0,q}_\infty=E^{0,q}_{q+1}\subseteq E^{0,q}_2$ is equal to the image of the natural map 
\begin{equation}\label{eq-Fbclass}
f_b^*: H^q(X, \bR)\rightarrow H^q(X_b, \bR)^{\mathrm{inv}}=E^{0,q}_2.
\end{equation}
Note that $f_b^*$ coincides with the composition $H^q(X, \bR)\rightarrow Gr_{\cF_L}^0 H^q(X, \bR)\rightarrow E^{0,q}_2$. 
For our later purpose, we need to better understand the 1st and 2nd pages of the Leray spectral sequence.
For any $\eta=\sum_{I} du_I\wedge \eta_I \in \cF_L^{p}\cE^{p+q}(X)$ as in \eqref{eq-etalc}, we know that
for any $I$ with $|I|=p$, $\eta_I|_{X_b}$ is a closed differential form on any fibre ${X_b}$, i.e. $(d\eta_I)|_{X_b}=0$. So $[\eta_I|_{X_b}]$ defines a class in $V^q_b=H^q(X_b, \bR)$. Then the space $E_1^{p,q}$ has the following description:
\begin{eqnarray*}
 \cE^p(B, V^q)&=&\{[\eta]=[\eta]_p:=\sum_{|I|=p} du_I\otimes [\eta_I|_{X_b}]:\eta \in \cF_L^p \cE^{p+q}(X), \\
&& \hskip 3cm
  \textrm{ for any $I$ with $|I|=p$}, \eta_{I}|_F \textrm{ is closed} \}
\end{eqnarray*}
which is the space of $V^q$-valued differential forms of degree $p$. 
Next we want to relate the $d_1$-differential to the following differential operator induced by the Gauss-Manin connection.
\begin{equation}\label{eq-defD}
\def\arraystretch{1.5}
\begin{array}{rcl}
D: \cE^p(B, V^q)&\rightarrow& \cE^{p+1}(B, V^q).\\
{D[\eta]_p} & = & \sum\limits_{\substack{|I|=p\\ i\in \{1,\dots, \dim_{\bR}B\}}} (-1)^{p} \left(du_I \wedge du_i \otimes [\mathcal{L}_{\partial_{u_i}} \eta_I ]\right).
\end{array}
\end{equation}
Here $\mathcal{L}_{\partial_{u_i}}$ is the Lie derivative with respect to the vector field $\partial_{u_i}$ for $i\in\{1,\dots, m\}$. Here we lift $\partial_{u_i}$ to be vector field defined on $\pi^{-1}(U)$ under the diffeomorphism in \eqref{eq-lctriv}.  

Using the vanishing $(d\eta_I)|_{X_b}=0$, it is easy to see that the graded piece in $d\eta\in \cF_L^{p+1}/\cF_L^{p+2}$ is represented by the terms:
\begin{equation*}
\sum_{\substack{|I|=p;\\ i\in \{1,\dots, \dim_{\bR}B\}}} (-1)^p du_I\wedge du_i\wedge \mathcal{L}_{\partial_{u_i}}\eta_I+\sum_{|J|=p+1}(-1)^{p+1} du_J\wedge d\eta_J. 
\end{equation*}
Note that $d\eta_J$ in the second term is exact when restricted to any fiber. In particular, for any $\eta\in \cF_L^p\cE^{p+q}(X)$, we have:
\begin{equation}\label{eq-Dvsd}
d\eta\in \cF_L^{p+1}\cE^{p+q+1}(X) \quad \text{and} \quad 
D[\eta]_p=[d\eta]_{p+1}\in \cE^{p+1}(B, V^q).
\end{equation}

\begin{lem}\label{lem-d1vsD}
There is a natural surjective morphism:
\begin{equation}\label{eq-isom1}
\phi: \{a\in \cF_L^p \cE^{p+q}(X): da\in \cF_L^{p+2}\cE^{p+q+1}(X) \}\rightarrow \mathrm{Ker}(D: \cE^p(B, V^q)\rightarrow \cE^{p+1}(B, V^q)).
\end{equation}
which induces an isomorphism 
\begin{equation}\label{eq-isom3}
E_2^{p,q}\cong \frac{\mathrm{Ker}(D: \cE^p(B, V^q)\rightarrow \cE^{p+1}(B, V^q))}{\mathrm{Im}(D: \cE^{p-1}(B, V^q)\rightarrow \cE^{p}(B, V^q))}=H^p(B, V^q).
\end{equation}
\end{lem}
\begin{proof}
For simplicity of notation, denote the left-hand-side and the right-hand-side of \eqref{eq-isom1} by $\mathcal{V}_1$ and $\mathcal{V}_2$ respectively. We also denote $\ell=p+q$. 

If $\eta\in \mathcal{V}_1$, then $d\eta\in \cF_L^{p+2}\cE^{\ell+1}(X)$ and hence $D[\eta]_p=[d\eta]_{p+1}=0$. So we get the natural morphism $\phi$ that maps $\eta$ to $[\eta]_p$. 
To see that this map is surjective, choose any $[\eta]\in \mathcal{V}_2$ with $\eta\in \cF_L^p\cE^\ell(X)$ such that $\eta_I|_{X_b}$ is closed for any fibre $X_b$. Then $[d\eta]_{p+1}=D[\eta]_p=0$ implies that $(d\eta)_J|_{X_b}$ is exact for any $J$ with $|J|=p+1$ and any fiber $X_b$. Choose a coordinate system $\{(U_k, \{u^{(k)}_i\}\}$ of $B$ and a partition of unity $\{\rho_k\}$ subordinate to the covering $\{U_k\}$. For each $b\in U_k$ and any $J$ with $|J|=p+1$, we can find $\chi_{J,b}$ on $X_b$ such that $(d\eta)_J|_{X_b}=d \chi^{(k)}_{J,b}$ and $\chi^{(k)}_{J,b}$ depends smoothly on $b\in U_k$ for any fixed $U_k$. Because any fiber is diffeomorphic to a fixed smooth manifold $F$, the family $\{\chi^{(k)}_{J,b}; b\in U_k\}$ can be considered as a smooth form $\chi^{(k)}_J$ on $\pi^{-1}(U_k)\cong U_k\times F$ (a smooth diffeomorphism). 
Define 
\begin{equation*}
\eta'=\eta-\sum_k \rho_k\sum_{|J|=p+1} (-1)^{p+1} du_J^{(k)}\wedge \chi^{(k)}_J. 
\end{equation*} 
Then $\eta'$ satisfies $\eta'-\eta\in \cF_L^{p+1}\cE^{p+q}$ and $(d\eta')_{J}|_{X_b}=0$ for any $J$ with $|J|=p+1$ and for any $b\in B$. In particular $\eta'\in \mathcal{V}_1 $ and satisfies $\phi(\eta')=[\eta']=[\eta]$. 

To see that $\phi$ induces the isomorphism \eqref{eq-isom3}, we need identify the kernel of $\phi$. So assume that $[\eta]=D[\chi]\in \mathrm{Im}(D: \cE^{p-1}(B, V^q)\rightarrow \cE^p(B, V^q))$. Then we can assume $\eta\in \cF_L^p\cE^{\ell}(X)$ with $d\eta\in \cF_L^{p+1}\cE^{\ell+1}(X)$ and $\chi\in \cF^{p-1}_L\cE^{\ell-1}(X)$. Then 
\begin{equation*}
[\eta]_{p}=D[\chi]_{p-1}=\sum_{|I|=p-1}(-1)^{p-1}(du_I\wedge du_i \otimes [\mathcal{L}_{\partial_{u_i}}\chi_I])=[d\chi]_{p}.
\end{equation*}
This implies $[\eta-d\chi]_p=0$ which implies $\eta\in \mathcal{V}'_1:=d \cF_L^{p-1}\cE^{\ell-1}(X)+\cF_L^{p+1}\cE^{\ell}(X)$. 
So we get the identity $\phi^{-1}(\mathrm{Im}(D))=\mathcal{V}'_1\cap \mathcal{V}_1$.
By the definition in \eqref{eq-Erpq}, $E^{p,q}_{2}=\frac{\mathcal{V}_1}{\mathcal{V}'_1\cap \mathcal{V}_1}$. So we conclude that $\phi$ induces an isomorphism.

\end{proof}

\subsection{K\"{a}hler case}

From now on, we assume that the fibers $\{X_b; b\in B\}$ and the base $B$ are all K\"{a}hler manifolds. 
Suppose that we have $\alpha\in H^0(B, V^2)\cong E^{0,2}_2$ that restricts to be a K\"{a}hler class on each fibre. 
The main problem we are considering is whether we can find a closed $(1,1)$-form $Q$ on $X$ such that $[Q]|_{X_b}=f_b^*[Q]=\alpha$, i.e. $\alpha$ is contained in the image of the natural morphism in \eqref{eq-Fbclass} for $q=2$. 
According to Theorem \ref{thm-main2}, the answer is positive if and only if $d_2(\alpha)=0$. 
In some sense, Theorem \ref{thm-main2} is a converse to the following basic result in K\"{a}hler geometry by Blanchard and Deligne:
\begin{thm}[\cite{Bla56, Del68}]
Let $X$ and $B$ be compact K\"{a}hler manifolds and $\pi: X\rightarrow B$ be a holomorphic submersion, then the Leray spectral sequence degenerates at the $E_2$-page. 
\end{thm}
In section \ref{sec-HS} we will use Deligne's method for proving this theorem to prove an extension result for Hermitian-Symplectic submersions (see Proposition \ref{prop-HSd2}). 

Another interesting consequence of Theorem \ref{thm-main1} is the following positive answer to a question of Harvey-Lawson in \cite{HL83}.
\begin{thm}[Corollary of Theorem \ref{thm-main1}]\label{cor-HL}
Assume that the fiber of the submersion has dimension 1. Then there exists a K\"{a}hler metric on $X$ if and only if the homology class of any fiber of $X$ is not zero. 
\end{thm}
In fact, when $\dim X_b=1$. 
$H^2(X_b, \bR)^{\mathrm{inv}}=H^{1,1}(X_b)=\bR$ induced by the natural orientation of the fibers. $\mathbf{(Condition\; I)}$ is satisfies if and only if the the map $f_b^*$ in \eqref{eq-Fbclass} is non-zero, which by duality is equivalent to the non-vanishing of the Gysin map $(f_b)_!: H^{0}(X_b,\bR)\cong H_2(X_b, \bR)\rightarrow H_{2}(X, \bR)\cong H^{2n-2}(X, \bR)$. This is also equivalent to the non-vanishing of the class $cl(X_b):=(f_b)_! [1]$ that is defined as the cohomology class of the fiber. On the other hand the complex subvariety $X_b$ defines the current $\{X_b\}$ of integration along $X_b$. The current $\{X_b\}$ is a closed positive current of type $(n-1, n-1)$ and its de Rham cohomology class is $cl(X_b)\in H^{n-1, n-1}(X, \bR)\subset H^{2n-2}(X, \bR)$. 
If there exists a K\"{a}hler metric on $X$, then it is to see that the current $\{X_b\}$ is not $(n-1,n-1)$-component an exact current. 
These statements together easily imply that the following equivalences, which imply \cite[Theorem 1.7]{HL83} and answer \cite[Note D]{HL83}.
\begin{equation*}
\xymatrix{
cl(X_b) \neq 0 \ar@{=>}[dr] & & \ar@{=>}[ll] \{X_b\} \text{ is not the $(n-1,n-1)$-component of an exact current }\\
& \text{K\"{a}hler}  \ar@{=>}[ur] &}
\end{equation*}



\subsection{Deligne-Hodge theory for polarized variation of Hodge structures}\label{sec-Deligne}

From now on we assume that $\mathbf{(Condition\; Ia)}$ holds true. Then the local system $V^q=R^q \pi_*\bR$ underlies a polarized variation of Hodge structures. 
In our construction of K\"{a}hler metrics for holomorphic submersions, we will need Deligne's Hodge theory for $H^p(B, V^q)$ to derive a variant of the $\partial\bar{\partial}$-lemma (Lemma \ref{lem-D'D''}). So we briefly recall this theory following Zucker's presentation in \cite{Zuc79}. 

For any $b\in B$, because the fiber $X_b$ is K\"{a}hler, we have the Hodge filtration $\{\mathscr{F}^{\ell}:=\mathscr{F}^{\ell}H^q(X_b, \bC)\}$ of $V^q_{\bC,b}:=H^q(X_b, \bC)$ that induces the Hodge decomposition for each fiber:
\begin{equation}\label{eq-fiberHodge}
V^q_{\bC,b}=H^{q}(X_b, \bC)=\bigoplus_{\ell=0}^q \mathcal{H}^{\ell, q-\ell}(X_b) \quad \text{ with } \quad \mathcal{H}^{\ell, q-\ell}(X_b)=\mathscr{F}^{\ell}\cap \overline{\mathscr{F}^{q-\ell}}. 
\end{equation}
This fiberwise Hodge decomposition in \eqref{eq-fiberHodge} defines smooth subbundles $\mathcal{H}^{\ell, q-\ell}$ of the smooth vector bundles associated to $V^q_{\bC}=V^q\otimes_{\bR}\mathbb{C}$.  If $\cE^p(V^q)=\Gamma(B, V^q_{\bC}\otimes_{\bC}\cE^p_B)$ be the smooth $V^q_{\bC}$-valued $p$-forms on $B$, then we have the decompositions:
\begin{eqnarray*}
\cE^0(V^q)=\bigoplus_{\ell=0}^q \cE^0(\mathcal{H}^{\ell, q-\ell}), \quad \cE^p(V^q)=\bigoplus_{k,\ell}\cE^{k, p-k}(\mathcal{H}^{\ell,q-\ell});\\
\cE^p (V^q)=\bigoplus_{R,S}\cE^p (V^q)^{R,S}, \quad \cE^{p}(V^q)^{R,S}=\bigoplus_{\substack{a+c=R, b+d=S\\ a+b=p, c+d=q}}\cE^{a,b}(\mcH^{c,d}).
\end{eqnarray*}
We can decompose the Gauss-Manin connection $D$ as two conjugate operators:
\begin{equation*}
D=\nabla\oplus \overline{\nabla}: \cE^0(V^q)\rightarrow \cE^{1,0}(V^q)\oplus \cE^{0,1}(V^q).
\end{equation*}
The flatness of $D$ implies that the operator $\overline{\nabla}$ satisfies the integrability condition $\overline{\nabla}^2=0$. So $\overline{\nabla}$ defines a holomorphic structure that gives rise to the associated holomorphic vector bundle $\mathscr{V}^q=V^q\otimes_{\bR}\mathcal{O}_B$.
The Hodge filtration induces a filtration of holomorphic sub-bundles $\{\mathscr{F}^{\ell}\mathscr{V}^q\}$ of $\mathscr{V}^q$. Griffiths proved that the holomorphic Gauss-Manin connection $\nabla$ satisfies the infinitesimal period relation (also called Griffiths transversality): 
\begin{equation}\label{eq-transv}
\nabla \mathscr{F}^{\ell}\mathscr{V}^q\subseteq \Omega_B^1\otimes \mathscr{F}^{\ell-1}\mathscr{V}^q. 
\end{equation}
The relation in \eqref{eq-transv} 
implies the inclusion:
\begin{equation*}
D \cE^{a,b}(\mcH^{c,d})\subseteq \cE^{a+1,b}(\mcH^{c,d})\oplus \cE^{a+1,b}(\mcH^{c-1,d+1})\oplus \cE^{a,b+1}(\mcH^{c+1,d-1}) \oplus \cE^{a, b+1}(\mcH^{c,d}).
\end{equation*}
Correspondingly, the operator $D$ decomposes as:
\begin{equation}\label{eq-Ddecomp}
\begin{array}{ll}
\partial': & \cE^{a,b}(\mcH^{c,d})\rightarrow \cE^{a+1,b}(\mcH^{c,d})\\
\theta:  & \cE^{a,b}(\mcH^{c,d})\rightarrow \cE^{a+1,b}(\mcH^{c-1,d+1})\\
\bar{\partial}': & \cE^{a,b}(\mcH^{c,d})\rightarrow \cE^{a,b+1}(\mcH^{c,d})\\
\bar{\theta}: & \cE^{a,b}(\mcH^{c,d})\rightarrow \cE^{a,b+1}(\mcH^{c+1,d-1}).
\end{array}
\end{equation}
It is well-known that $\theta$ is given by the wedge product with the Kodaira-Spencer class of the associated deformation.
More precisely, for any $b\in B$ and any $v\in T_bB$, if we denote by $\tau_b(v)\in H^1(X_b, T X_b)$ the Kodaira-Spencer class of the deformation induced by the holomorphic submersion in the direction $v$, the cup-product with $\tau_b(v)$ induces the fiberwise map:
\begin{eqnarray}\label{eq-KS}
\tau_b(v): H^{d}(X_b, \Omega_{X_b}^c)\rightarrow H^{d+1}(X_b, \Omega_{X_b}^{c-1}).
\end{eqnarray}
Then $\theta$ is the natural tensorial map induced by the family of maps $\{\tau_b\in T_{X_b}^{*(1,0)}\otimes (\mathcal{H}_b^{c,d})^*\otimes \otimes \mathcal{H}_b^{c-1,d+1}\}$.  

With the decomposition in \eqref{eq-Ddecomp}, we can define the following operator:
\begin{eqnarray*}
&&D'=\partial'+\bar{\theta}: \cE^{p}(V^q)^{R,S}\rightarrow \cE^{p+1}(V^q)^{R+1,S}, \\
&&D''=\bar{\partial}'+\theta: \cE^{p}(V^q)^{R,S}\rightarrow \cE^{p+1}(V^q)^{R, S+1}. 
\end{eqnarray*}
The above discussion applies to any variation of Hodge structures of any fixed weight. So for the rest of this subsection, we just write $V$ for $V^q$. 
Now assume that $V$ is polarizable: there exists a non-degenerate flat bilinear pairing on $V$ defined over $\bR$ that satisfies the Hodge-Riemann bilinear relations (see \cite[vI.10.1]{Voi02} for a precise definition). Then there is a flat Hodge metric on the smooth vector bundle $\cE^{\cdot}(V)$ which allows us to define the adjoint operators $D^*, D'^*, D''^*$ and the Laplace operators:
\begin{equation*}
\square_{D}=DD^*+D^*D, \quad \square_{D'}=D'D'^*+D'^*D', \quad \square_{D''}=D''D''^*+D''^*D. 
\end{equation*}
Denote by $\mathfrak{h}=\mathfrak{h}^{p}(V):=\mathrm{Ker}(\square_D)$ the space of harmonic form in $\cE^{p}(V)$. Then the standard Hodge theory gives us:
\begin{thm}[Hodge theorem]\label{thm-Hodge}
There is an isomorphism $H^p(B, V)\cong \mathfrak{h}$. 
\end{thm}
Deligne proved the following results:
\begin{thm}[{Deligne, see \cite[Section 2]{Zuc79}}]\label{thm-Deligne}
\begin{enumerate}
\item The following generalized K\"{a}hler identities are true. 
\begin{equation}\label{eq-KahlerId}
[\Lambda, D]=-C^{-1}D^*C, \quad [\Lambda, D'']=-\sqrt{-1}D'^*, \quad [\Lambda, D']=\sqrt{-1} D''^*
\end{equation}
where $C$ is the operator that is the direct sum of the scalar operator $i^{R-S}$ on $\cE^p(V)^{R,S}$. 
\item
The following equalities for Laplacians hold true:
\begin{equation*}
\square_{D'}=\square_{D''}, \quad \square_{D}=\square_{D'}+\square_{D''}=2 \square_{D''}.
\end{equation*}
As a consequence, a form is harmonic if and only if all of its $(R, S)$-components are harmonic. 
\item Denote by $\mathfrak{h}^{R,S}$ the space of harmonic form in $\cE^{p}(V)^{R,S}$. There is a Hodge structure on $H^p(B, V)$ induced by the isomorphism in Theorem \ref{thm-Hodge} and the decomposition:
\begin{equation*}
\mathfrak{h}=\bigoplus_{R+S=p+q} \mathfrak{h}^{R,S}.
\end{equation*}
\end{enumerate}
\end{thm}
Note that Deligne's theory reduces to the classical Hodge theory when $V=\bR$. 
As in the classical case, the K\"{a}hler identities imply the following useful identity:
\begin{eqnarray}\label{eq-D'D''*}
D'D''^*+D''^*D'=0.
\end{eqnarray}
We will need the following variant of the standard $\partial\bar{\partial}$-lemma and a Lemma by Blanchard (\cite[Lemma II 3.2]{Bla56}).
\begin{lem}\label{lem-D'D''}
\begin{enumerate}
\item
If $x \in \cE^\cdot(V)^{R,S}$ is $D$-exact, then there exists $z \in \cE^{\cdot}(V)^{R-1,S-1}$ such that $x=D'D'' z$.
\item
If $x\in \cE^\cdot(V)^{R+1,S}\oplus \cE^\cdot(V)^{R, S+1}$ is $D$-exact, then there exists $w\in \cE^{\cdot}(V)^{R,S}$ such that $x=D w$. 
\end{enumerate}
\end{lem} 
We provide the short proof, which is the same as in the classical case, for the reader's convenience. 
\begin{proof}
By standard Hodge theorem, there is a Green operator $G: \cE^{\cdot}(V)\rightarrow \cE^{\cdot}(V)$ that preserves types, commutes with $D'', D'$ and satisfy $G\square_{D}=\square_{D}G=\mathrm{Id}-\mathbb{H}$ where $\mathbb{H}$ is the orthogonal projection to the space of harmonic forms. So we get a decomposition for $x:=[\eta]$:
\begin{equation*}
x=\mathbb{H}x+2(D''D''^*+D''^*D'')G x=2 D'' D''^*Gx. 
\end{equation*}
So we can find a form $y=2 D''^*Gx$ of type $(R, S-1)$ such that $x=D''y$. Applying the similar decomposition for $y$ we get:
$
y=\mathbb{H}y+2(D'D'^*+D'^*D')Gy=D'z+D'^*u
$ (with $z=2D'^*Gy$ and $u=2 D'Gy$) and hence $x=D''D'z+D''D'^*u$. On the other hand, we know that $D'x=0$. So that $0=D'D''D'^*u=-D'D'^*D'' u$ where we used \eqref{eq-D'D''*}. Pairing this with $D''u$, we see that $D''D'^*u=0$. So we conclude $x=D''D'z$.  

For the second statement, set $x=y+z$ with $y\in \cE(V)^{R+1,S}$ and $z\in \cE(V)^{R,S+1}$. Because $\mathbb{H}x=0$, we get $\mathbb{H}y=\mathbb{H}z=0$. Because we also have $D'y=0$, there is a decomposition $y=D'D'^*Gy=D'u$ where $u=D'^*Gy$ is of type $(R,S)$. Then $x-Du=(x-D'u)-D''u=z-D''u$ is $D$-exact and is of type $(R, S+1)$. By the first statement, we get $x-Du=D''D'v=DD'v$ for $v\in \cE^{\cdot}(V)^{R-1,S}$. So the form $u+D'v$ is of type $(R,S)$ and satisfies the conclusion.

\end{proof}

\section{Proof of main results}\label{sec-main}






In this section, we prove Theorem \ref{thm-main2}, which also implies Theorem \ref{thm-main1}. The ``only if" direction is clear by the discussion in the introduction and section \ref{sec-Leray}. So we only need to prove the ``if" direction. 

In the following construction, we fix an open covering $\{U_k\}$ of $B$ with each $U_k$ biholomorphic to a polydisc and also a partition of unity $\{\rho_k\}$ subordinate to the covering $\{U_k\}$. 

By assumption $\mathbf{(Condition\; Ia)}$, there is an element $\sigma\in H^0(B, V^2)$ such that $\sigma|_{X_b}$ is a K\"{a}hler class for any $b\in B$.  
By refining the open covering $\{U_k\}$, we can assume that $\pi^{-1}(U_k)$ admits a K\"{a}hler metric $\omega_k$ and $[\omega_k]|_{X_b}=\sigma|_{X_b}$. Set $\omega=\sum_k \rho_k \omega_k$. 
Then $\omega$ is a globally defined $(1,1)$-form satisfying $\sigma=[\omega]\in H^0(B, V^2)=E^{0,2}_2$. 

To do concrete calculation,  we choose local holomorphic coordinate $\{w_\alpha, z_k\}=\{w^{(k)}_\alpha, z^{(k)}_r\}$ on $\pi^{-1}(U_k)$ such that $\pi$ is given by the map $\pi(w_\alpha, z_r)=(w_\alpha)$. Then $\omega$ can be expanded locally as:
\begin{equation*}\label{eq-omglocal}
\omega=a_{r\bar{s}}dz_r\wedge d\bar{z}_s+a_{\alpha \bar{r}}dw_\alpha \wedge d\bar{z}_r+a_{\bar{\beta}s}d \bar{w}_\beta \wedge dz_s+a_{\alpha\bar{\beta}} dw_\alpha\wedge d\bar{w}_\beta.  
\end{equation*} 
With the assumption that $d_2[\omega]=0$, our goal is to correct $\omega$ by a $(1,1)$-form $\chi$ such that $Q:=\omega-\chi$ 
satisfies $d Q=0$ globally on $X$. 

We will achieve this in several steps by achieving $dQ\in \cF_L^p(\cE^3(X))$ for $p=1,2,3,4$ sub-sequentially (note that $\cF_L^4(\cE^3(X))=0$). 
For $p=1$ this is already true because $\omega|_{X_b}$ is closed. For $p=2$, we will see that the correction comes from the fact that
$d_1[\omega]=0$ so that $[\omega]\in E^{0,2}_2=H^0(B, V^2)$. For $p=3$, we need to use the assumption of $\mathbf{(Condition\; Ib)}$ that $d_2[\omega]=0$ and crucially the Lemma \ref{lem-D'D''}.  Finally one can complete the $p=4$ case with the usual $\partial\bar{\partial}$-lemma.

For simplicity of notation, we will denote by $(dQ)_{p;3-p}\in \cF_L^p \cE^3(X)/\cF_L^{p+1}\cE^3(X)$ the graded piece of $dQ\in \cF_L^p \cE^3(X)$. In other words, if $dQ\in \cF_L^p \cE^3(X)$ then $(dQ)_{p;3-p}=0$ if and only if $dQ\in \cF_L^{p+1}\cE^3(X)$. Note that this is the case if and only if 
for any $b\in B$, $f_b^*\iota_{\tilde{v}_{\alpha_1}}\cdots \iota_{\tilde{v}_{\alpha_p}}dQ=0$ where $\tilde{v}_{\alpha_i}\in \{v_1, \dots, v_m, \bar{v}_1,\dots, \bar{v}_m\}$ where $m=\dim B$, $v_{\alpha}$ (resp. $\bar{v}_\alpha$) are lifts of local coordinate vector fields $\partial_{w_\alpha}$ (resp. $\partial_{\bar{w}_\alpha}$) near $b\in B$ and $f_b: X_b\rightarrow X$ is the embedding map of fiber. 
 
\begin{itemize}
\item \textbf{ (Step 1: Elimination of $(d\omega)_{0;3}$)} This is already true for the $(1,1)$-form $\omega$ because of the vanishing (recall that $f_b: X_b\rightarrow X$ is the embedding of the fiber)
\begin{equation*}
f_b^* d \omega=d f_b^*\omega=0.
\end{equation*}

\item \textbf{ (Step 2: Elimination of $(d \omega)_{1;2}$)} We can calculate:
\begin{eqnarray}\label{eq-1v}
f_b^* \iota_{v_\alpha} d{\omega}&=&f_b^* (\mathcal{L}_{v_\alpha} \omega-d \iota_{v_\alpha}\omega)=f_b^* \mathcal{L}_{v_\alpha} \omega -d f_b^*\iota_{v_\alpha}\omega.
\end{eqnarray}
Because $d_1[\omega]=[d\omega]_2=0$, $f_b^*\mathcal{L}_{v_\alpha}\omega$ is an exact form for any $b\in B$. So the right-hand-side of \eqref{eq-1v} is an exact form of type $(1,1)+(0,2)$ for each $b\in B$. By Lemma \ref{lem-D'D''}.2, there 
there exists a form $\bar{h}_{\alpha}$ of fiber type $(0,1)$ such that $d\bar{h}_\alpha=f_b^* \iota_{v_\alpha} d{\omega}$.
Moreover we can assume that over any $U_k$, $h_\alpha=h^{(k)}_\alpha$ depends smoothly on $b\in U_k$. Set $\chi^{(k)}=dw_\alpha\wedge \bar{h}_{\alpha}+d\bar{w}_\alpha\wedge h_{\alpha}$ over $\pi^{-1}(U_k)$ and define: 
\begin{equation*}
Q=\omega-\sum_k \rho_k \chi^{(k)}.
\end{equation*}
Then $Q$ is globally defined and satisfies $dQ\in \cF_L^2\cE^3(X)$. From now on, replace $\omega$ by $Q$ (i.e. we call $Q$ our new $\omega$). 

\item\textbf{(Step 3: Elimination of the components of $(d \omega)_{2,0;1,0}$)}  
 Since $d\omega\in \cF_L^2\cE^3(X)$, it defines a smooth section $[d\omega]$ of $\cE^2(V^1)$ which satisfies $D[d\omega]=0$ and 
$d_2 [\omega]=[d\omega]\in E^{2,1}_2=H^2(B, V^1)$. Because the complex structure is integrable, $d\omega$ is a form of type $(2,1)+(1,2)$ and $[d\omega]\in \cE^2(V^1)$ is contained in 
\begin{eqnarray*}
\cE^{2,0}(\mcH^{0,1})\oplus \cE^{1,1}(\mcH^{1,0})\oplus \cE^{0,2}(\mcH^{1,0})\oplus \cE^{1,1}(\mcH^{0,1}).
\end{eqnarray*}
Now we use the assumption that $ d_2  [\omega]=0$ so that $[d\omega]_2\in \cE^2(V^1)=\cE^{\cdot}(V)^{2,1}\oplus \cE^{\cdot}(V)^{1,2}$ is $D$-exact. 
By Lemma \ref{lem-D'D''}.2, there exists 
$[\chi]\in \cE^{\cdot}(V)^{1,1}$ such that $[d\omega]_2=D[\chi]=[d\chi]_2$.  
Because $[\chi]$ belongs to $\cE^{\cdot}(V)^{1,1}=\cE^{1,0}(\mcH^{0,1})\oplus \cE^{0,1}(\mcH^{1,0})$, 
we can construct a representative $\chi$ as a $(1,1)$-form (by using partition of unity as in the proof of Lemma \ref{lem-d1vsD}). 

The terms of $d\omega$ with type $(0,2;1,0)$ or $(2,0; 0,1)$ are given by: 
\begin{eqnarray*}
&&\frac{1}{2}((\partial_{\bar{\alpha}}a_{\bar{\beta}r}-\partial_{\bar{\beta}}a_{\bar{\alpha}r})dz_r)d \bar{w}_\alpha\wedge d \bar{w}_\beta+\frac{1}{2}((\partial_{w_\alpha} a_{\beta \bar{s}}-\partial_{w_\alpha} a_{\beta \bar{s}})d\bar{z}_s)\wedge dw_\alpha\wedge dw_\beta\\
&=& \frac{1}{2} P_{\bar{\alpha}\bar{\beta}} d\bar{w}_\alpha\wedge  d\bar{w}_{\beta}+\frac{1}{2}\overline{P_{\bar{\alpha}\bar{\beta}}}\wedge dw_\alpha\wedge dw_\beta. 
\end{eqnarray*} 
By replacing $\omega$ by $\omega-\chi$, we can assume that $[d\omega]_2=0\in \cE^{2}(B, V^1)$. 
This implies that over each fiber $X_b$, the (1,0)-form $P_{\bar{\alpha}\bar{\beta}}$ is an exact form. Because $X_b$ is K\"{a}hler, this implies that $P_{\bar{\alpha}\bar{\beta}}=0$. So $d\omega$ does not have terms of type $(0,2;1,0)$ or $(2,0;0,1)$. 

\item \textbf{(Step 4: Elimination of components of $(d \omega)_{1,1;1,0}$)}

 The terms of $d\omega$ with type $(1,1;1,0)$ or $(1,1; 0,1)$ can be written as $P_{\alpha\bar{\beta}}\wedge dw_\alpha\wedge d\bar{w}_\beta$ where 
\begin{eqnarray*}
P_{\alpha\bar{\beta}}&=&(\partial_{\bar{w}_\beta}a_{\alpha \bar{s}}) d\bar{z}_s-(\partial_{w_\alpha}a_{\bar{\beta}r})d{z}_r+(\partial_{\bar{z}_s}a_{\alpha\bar{\beta}}) d\bar{z}_s+(\partial_{z_r}a_{\alpha\bar{\beta}}) dz_r.
\end{eqnarray*}
As in Step 3, we can assume that $[d\omega]_2=0\in \cE^{2}(B, V^1)$, which again means that $P_{\alpha\bar{\beta}}$ in the above expression is exact. 
So for any $b\in U_k$, there exists a smooth function $\xi^{(k)}_{\alpha\bar{\beta}}$ on $X_b$ such that $P_{\alpha\bar{\beta}}=d \xi^{(k)}_{\alpha\bar{\beta}}$ and $\int_{X_b}\xi^{(k)}_{\alpha\bar{\beta}}\omega^n=0$. Moreover, we can assume that $\xi^{(k)}_{\alpha\bar{\beta}}$ depends smoothly on $b\in U_k$. Define a new $(1,1)$-form
\begin{equation*}
Q=\omega-\sum_k \rho_k\cdot \xi^{(k)}_{\alpha\bar{\beta}} dw^{(k)}_\alpha\wedge d\bar{w}^{(k)}_\beta. 
\end{equation*}
Then $Q$ satisfies $dQ\in \cF_L^3\cE^3(X)$.  
Now we replace $\omega$ by $Q$. 

\item (\textbf{Step 5: Completion of construction})
Because $d \omega\in \cF_L^3\cE^3(X)$, we can get:
\begin{equation*}
d\omega= f_{\alpha\beta\bar{\gamma}} dw_\alpha\wedge d {w}_\beta\wedge d\bar{w}_\gamma+ \overline{f_{\alpha\beta\bar{\gamma}}} d\bar{w}_\alpha\wedge d \bar{w}_\beta\wedge d {w}_\gamma.
\end{equation*} 

$dd\omega=0$ implies that
$f_{\alpha\beta\bar{\gamma}}$ are constant along fibers. So $d\omega=\pi^*\eta_B$ for some $\eta_B\in \cE^{2,1}(B)+\cE^{1,2}(B)$. Clearly $\eta_B$ is closed. $\eta_B$ must be exact because for any closed $(m-2)$-form $T$ on $B$
\begin{equation*}
\int_B \eta_B\wedge T=\int_X \omega^n\wedge \pi^*\eta_B\wedge \pi^*T=\frac{1}{n+1}\int_X d \omega^{n+1}\wedge \pi^*T=0.
\end{equation*}
Because $B$ is assumed to be K\"{a}hler $\mathbf{(Condition\; II)}$, Lemma \ref{lem-D'D''} applied to the constant local system $\bR$, we know that there is a smooth $(1,1)$-form such that $\eta_B=d\chi_B$. So the $(1,1)$-form $Q'=\omega-\pi^*\chi_B$ is a closed $(1,1)$-form that restricts to be K\"{a}hler forms on each fiber. Let $\omega_B$ be a K\"{a}hler form on $B$, then for $K\gg 1$, the closed $(1,1)$-form $Q:=Q'+K\cdot \pi^*\omega_B$ becomes a K\"{a}hler form on $X$.

\end{itemize}

\begin{rem}\label{rem-Blacons}
The above proof generalizes and simplifies the proof of \cite[TH\'{E}OR\'{E}M PRINCIPAL II]{Bla56} where Blanchard considered the special case when the holomorphic submersion $\pi: X\rightarrow B$ is isotrivial and and also satisfies an extra condition that $\pi_1(B)$ acts trivially on $H^1(F)$. 
In fact, Blanchard showed in \cite[TH\'{E}R\'{E}ME II.1.I]{Bla56} that under this triviality assumption and $\mathbf{(Condition\; Ia)}$, the $\mathbf{(Condition\; Ib)}$ is equivalent to the condition that the transgression map $H^1(F)\rightarrow H^2(B)$ is 0. Then he used this vanishing of transgression map to carry out the construction of K\"{a}hler metrics in his original steps corresponding to Step 3 and Step 4 above. Moreover because the analytic fiber bundle is isotrivial, the standard Hodge theory for a K\"{a}hler manifold is sufficient for \cite{Bla56} instead of Deligne's Hodge theory for polarized variation of Hodge structures used in the above proof. 

\end{rem}

\begin{rem}\label{rem-11case}
If we know that the class $[Q]$ contains a closed $(1,1)$-form $Q$, then the above construction can be greatly simplified as follows. We start with the $(1,1)$-form $\omega$ constructed by using partition of unity as at the beginning of this section. For any $b\in B$, we can use the standard $\partial\bar{\partial}$-lemma to conclude that there exists $\psi_b\in C^\infty(X_b, \bR)$ such that $\omega_b=Q|_{X_b}+\sqrt{-1}\partial\bar{\partial}\psi_b$ and $\int_{X_b} \psi_b \omega_b^n=0$. Moreover, it is easy to see that $\psi_b$ depends smoothly on $b\in B$. This allows us to define a smooth function $\Psi$ on $X$ such that $\Psi|_{X_b}=\psi_b$. Define $Q=\chi+\sqrt{-1}\partial\bar{\partial} \Psi$. Then $Q$ is a closed $(1,1)$-form that restricts to be a K\"{a}hler metric $\omega_b$ on $X_b$. Choosing a K\"{a}hler metric $\omega_B$ on $B$,  $Q+K \cdot \pi^*\omega_B$ is then a K\"{a}hler metric for $K\gg 1$. 

If we just know that $Q$ is a closed 2-form not necessarily of $(1,1)$-type, then a similar construction would only produces a Hermitian-Symplectic form on $X$. The Hermitian-Symplectic structure will be discussed in the next section. 
\end{rem}

\section{An application to Hermitian-Symplectic structures}\label{sec-HS}

Let $X$ be a complex manifold and $Q$ be a closed 2-form on $X$. 
$Q$ is a Hermitian-Symplectic (HS) structure precisely when $Q$ satisfies the following two conditions. 
\begin{enumerate}
\item $Q$ is a symplectic form. In other words, $Q^{\dim X}$ is non-vanishing and $dQ=0$.  
\item If $Q=Q^{2,0}+Q^{1,1}+Q^{0,2}$ is the decomposition of $Q$ into differential forms of type $(2,0)$, $(1,1)$ and $(0,2)$ respectively, then $Q^{1,1}$ is a positive definite $(1,1)$-form. 
\end{enumerate}
We are going to apply Theorem \ref{thm-main2} to prove the following general result.
\begin{prop}\label{prop-HSKahler}
Let $\pi: X\rightarrow B$ be a holomorphic submersion with K\"{a}hler fibers and a K\"{a}hler base. Assume that $X$ admits a Hermitian-Symplectic structure. 
Then the following conditions are equivalent:

\begin{enumerate}
\item $X$ is K\"{a}hler.
\item There exists $[\omega]\in H^0(B, R^2\pi_*\bR)$ that restricts to a K\"{a}hler class on each fiber $F$. 
\item The variation of Hodge structure $R^2\pi_*\bR$ is polarizable. 
\item $X$ satisfies the $\partial\bar{\partial}$-lemma. 
\end{enumerate}
\end{prop}

The rest of this section is devoted to proving Proposition \ref{prop-HSKahler} and deriving Theorem \ref{thm-HSfiber} as its corollary. 
 Let $F$ be any fiber of $\pi$. It is easy to see that $\chi=Q|_F$ is a Hermitian-Symplectic form on $F$. 
Because $F$ is assumed to be K\"{a}hler, we have a Hodge decomposition:
\begin{equation*}
H^2(F,\bC)=H^{2,0}(F)\oplus H^{1,1}(F)\oplus H^{0,2}(F), \quad H^{0,2}(F)=\overline{H^{2,0}(F)}.
\end{equation*}
The cohomology class $[\chi]\in H^2(F, \bC)$ decomposes as $[\chi]=[\chi]^{0,2}+[\chi]^{1,1}+[\chi]^{0,2}$ with respect to the above decomposition. On the contrast, if we denote by $\chi^{1,1}$ the $(1,1)$-component of the 2-form $\chi$, then $\chi^{1,1}$ is in general only $\partial\bar{\partial}$-closed but not $d$-closed. 

\begin{prop}[\cite{LZ09}]\label{prop-LZ}
 If $\chi$ is a Hermitian-Symplectic form on a K\"{a}hler manifold $F$, then 
$[\chi]^{1,1}$ is a K\"{a}hler class. 
\end{prop}
As pointed out in \cite{LZ09}, this is a consequence of deep results from \cite{DP04}. 
For the reader's convenience, we provide a quick proof. 
\begin{proof}
First one shows that $[\chi]^{1,1}$ is in the nef cone which is the closure of the K\"{a}hler cone. By \cite[Theorem 4.3]{DP04}, this is equivalent to the condition that for every irreducible analytic subset $Y\subseteq F$, $\dim=p$, and every K\"{a}hler class $[\omega_F]\in H^{1,1}(F,\bR)$, 
\begin{equation*}
\int_Y [\chi]^{1,1}\wedge \omega_F^{p-1}\ge 0.
\end{equation*}
But for reason of the types, the above integral is equal to 
\begin{eqnarray*}
&&\int_Y [\chi] \wedge \omega_F^{p-1}=\int_Y \chi\wedge \omega_F^{p-1}=\int_Y \chi^{1,1}\wedge \omega_F^{p-1}\ge 0.
\end{eqnarray*} 
For any $\alpha\in H^{1,1}(F, \bR)$ and $\epsilon$ sufficiently small, the class $[\chi]+\epsilon \alpha$ clearly also contains a Hermitian-Symplectic form and hence $[\chi]+\epsilon \alpha$ is also contained in the nef cone. So we conclude that $[\chi]^{1,1}$ is in the interior of the nef cone which is nothing but the K\"{a}hler cone. 

\end{proof}


To proceed, we consider the following definition.
\begin{defn}[{see \cite[Definition 4.96]{FOT08}}]
A symplectic manifold $(F, \chi)$ is of Lefschetz type if the multiplication by $[\chi]^{n-1}$ 
\begin{equation}\label{eq-Lefop}
L:=L_{[\chi]}: H^{1}(F, \bR)\rightarrow H^{2n-1}(F, \bR). 
\end{equation}
is an isomorphism.
\end{defn}

\begin{cor}
Let $F$ be a K\"{a}hler manifold. 
If $\chi$ is a Hermitian-Symplectic form on $F$, then $(F, \chi)$ is of Lefschetz type. 
\end{cor}

\begin{proof}
By Poincar\'{e} duality, $H^1(F, \bR)$ and $H^{2n-1}(F, \bR)$ have the same dimension. So it is enough to prove that $L_{[\chi]}: H^1(F, \bR)\rightarrow H^{2n-1}(F, \bR)$ is injective.

Because $F$ is K\"{a}hler, by the Hodge decomposition and Theorem \ref{prop-LZ}, there exists a K\"{a}hler form $\omega_F$ and a (closed) holomorphic 2-form $\eta$ such that $[\chi]=[\omega_F+\eta+\bar{\eta}]$. 
Moreover we have the Hodge decomposition: $H^1(F, \bC)=H^{1,0}(F)\oplus H^{0,1}(F)$. 
So any $\alpha\in H^1(F, \bR)$ can be represented by a sum $\xi+\bar{\xi}$ where $\xi\in H^0(F, \Omega^1_F)$ is a closed holomorphic 1-form. 

Assume that $L^{n-1}(\alpha)=0$. Then there exists an $(2n-2)$-form $\phi$ such that:
\begin{equation*}
(\omega_F+\eta+\bar{\eta})^{n-1}\wedge (\xi+\bar{\xi})=d\phi
\end{equation*}
From this, we get $(\omega_F+\eta+\bar{\eta})^{n-1}\wedge \xi\wedge \bar{\xi}=d\phi\wedge \bar{\xi}=d (\phi\wedge \bar{\xi})$.
Integrating both sides, we get:
\begin{equation*}
\int_F (\omega_F+\eta+\bar{\eta})^{n-1}\wedge (\sqrt{-1} \xi\wedge \bar{\xi})=\sqrt{-1} \int_F d(\phi\wedge \bar{\xi})=0.
\end{equation*} 
By using the positivity of $\omega_F$, it is easy to verify that the left-hand-side is greater than $\int_F \omega_F^{n-1}\wedge \sqrt{-1}\xi\wedge \bar{\xi}$ which implies that the $L^2$-norm of $\xi$ (with respect to $\omega_F$) is equal 0. So we conclude that $\xi=0$ and hence $\alpha=0$.  

\end{proof}

Let $(F, \chi)$ be any symplectic manifold. Motivated by the K\"{a}hler case, we set:
\begin{equation*}
H^2(F, \bR)_{\mathrm{prim}}=\mathrm{Ker}\left(L_{[\chi]}^{n-1}: H^2(F, \bR)\rightarrow H^{2n}(F, \bR)\right).
\end{equation*} 
Then it is easy to see that there is a direct sum decomposition:
\begin{equation*}
H^2(F, \bR)=H^2(F, \bR)=\bR[\chi]\oplus H^2(F, \bR)_{\mathrm{prim}}. 
\end{equation*}
Indeed, it is easy to verify that $\bR[\chi]\cap H^2(F, \bR)_{\mathrm{prim}}=\{0\}$ and any $\alpha \in H^2(F, \bR)$ decomposes as
\begin{equation*}
\alpha=\left(\frac{[\chi]^{n-1}\cdot \alpha}{[\chi]^n}[\chi]\right)+\left(\alpha-\frac{[\chi]^{n-1}\cdot \alpha}{[\chi]^n}[\chi]\right)\in \bR[\chi]+H^2(F, \bR)_{\mathrm{prim}}.
\end{equation*}
So we have a Lefschetz-type decomposition in degree 2 for any symplectic manifold (cf. \cite{AT14}).

\begin{prop}\label{prop-HSd2}
Let $\pi: X\rightarrow B$ be a holomorphic submersion. Assume that there is a closed form $Q$ on $X$ such that $(X_b, Q_b:= Q|_{X_b})$ is a symplectic manifold of Lefschetz type. Then any $x\in H^0(B, R^2\pi_*\bR)$ extends to a global class in $H^2(X,\bR)$. 
\end{prop}
\begin{proof}
By the discussion in section \ref{sec-Leray}, 
it is enough to prove that $d_2: E^{0,2}_2\rightarrow E^{2,1}_2$ are $d_3: E^{0,2}_3\rightarrow E^{3,0}_3$ are both 0. 
We will adapt Deligne's method for proving the degeneration of Leray spectral sequence at the $E_2$-page for projective submersions to the HS setting. 

By the same argument as in \cite[4.2.2]{Voi02}, because $Q$ is closed, it induces a morphism of local systems:
\begin{equation}
L:=[Q]\cup: R^k\pi_*\bR\rightarrow R^{k+2}\pi_*\bR. 
\end{equation}
which is equal to  $L_{[Q_b]}=[Q_b]\cup $ on the stalk at $b\in B$. For $k=1, 2$, we have relative Lefschetz decomposition:
\begin{eqnarray*}
&&R^1\pi_*\bR=(R^1\pi_*\bR)_{\mathrm{prim}}, \quad 
R^2\pi_*\bR=(L\cdot R^{0}\pi_*\bR) \oplus (R^2\pi_*\bR)_{\mathrm{prim}} 
\end{eqnarray*}
which induces the decomposition:
\begin{equation*}
H^0(B, R^2\pi_*\bR)=L\cdot H^0(B, R^0 \pi_*\bR)\oplus H^0(B, (R^2\pi_*\bR)_{\mathrm{prim}}).
\end{equation*}
Because $R^0\pi_*\bR=\bR$ is a constant local system, $L\cdot H^0(B, R^0\pi_*\bR)=\bR\cdot [Q]$ and it easy to see that $d_2: L\cdot H^0(B, R^0\pi_*\bR) \rightarrow H^2(B, R^1\pi_*\bR)$ is equal to 0. So to show $d_2=0$, it is enough to show that $d_2$ restricted to $H^0(B, (R^2\pi_*\bR)_{\mathrm{prim}})$ is equal to 0. 
Consider the following commutative diagram. 
\begin{equation}
\xymatrix{
H^0(B, (R^2\pi_*\bR)_{\mathrm{prim}}) \ar[d]_{d_2} \ar[r]^{L^{n-1}} & H^0(B, R^{2n}\pi_*\bR) \ar[d]^{d_2}\\
H^2(B, R^1\pi_*\bR) \ar[r]^{L^{n-1}}  & H^2(B, R^{2n-1}\pi_*\bR)}
\end{equation}
The top row is zero by the definition of $(R^2\pi_*\bR)_{\mathrm{prim}}$. The bottom row is an isomorphism by the assumption that $(X_b, Q_{X_b})$ is of Lefschetz type. So we see that $d_2$ on the left side is indeed 0.

By using similar argument, we can easily prove that $d_2: E^{1,1}_2\rightarrow E^{3,0}_2$ is also 0. So we get $E^{0,2}_3=E^{0,2}_2$ and $E^{3,0}_3=E^{3,0}_2$. Then we can use the same argument as above to prove that $d_3: E^{0,2}_3\rightarrow E^{3,0}_3=H^3(B, R^0\pi_*\bR)$ is also equal to 0. For example, we use the following commutative diagram to conclude that $d_3$ restricted to $E^{0,2}_3=H^0(B, (R^2\pi_*\bR)_{\mathrm{prim}})$ is equal to 0:
\begin{equation}
\xymatrix{
H^0(B, (R^2\pi_*\bR)_{\mathrm{prim}}) \ar[d]_{d_3} \ar[r]^{L^{n}=0} & H^0(B, R^{2n}\pi_*\bR) \ar[d]^{d_3}\\
H^3(B, R^0\pi_*\bR) \ar[r]^{L^{n}}_{\cong}  & H^2(B, R^{2n}\pi_*\bR)}.
\end{equation}
\end{proof}

\begin{proof}[Proof of Proposition \ref{prop-HSKahler}]
By Theorem \ref{thm-main2} and Proposition \ref{prop-HSd2}, we know that condition 1 and condition 2 are equivalent (under the Hermitian-Symplectic assumption), and it is well-known that they imply condition 3 and condition 4. 

We show that condition 3 implies condition 2. We know that $[Q]\in H^0(B, V^2)$ is a flat section of the local system $V^2=R^2\pi_*\bR$ with respect to the Gauss-Manin connection. If the associated variation of Hodge structure on $V^2$ is polarized (condition 3), then by Deligne's Hodge theory as explained in \ref{sec-Deligne} the $(1,1)$-component of $[Q]$ is also flat so that $[Q]^{1,1}\in H^0(B, V^2)$. By Proposition \ref{prop-LZ}, we know that $[Q]^{1,1}$ restricts to K\"{a}hler class on each fiber. But this means exactly condition 2. 

Finally we show that condition 4 implies $\mathbf{(Condition \; I)}$ of Theorem \ref{thm-main1}. 
If $X$ satisfies the $\partial\bar{\partial}$-lemma, then by the well-known result of \cite{DGMS} $X$ admits a Hodge decomposition and the restriction morphism $f_b^*: H^2(X, \bR)\rightarrow H^2(X_b, \bR)$ is a morphism of Hodge structures. 
Let $[Q]^{1,1}$ be the $(1,1)$-component of $[Q]$ in the Hodge decomposition of $H^2(X,\bC)$. Then $[Q]^{1,1}$ is represented by a closed $(1,1)$-form and it restricts to become the $(1,1)$-component of $[Q|_{X_b}]\in H^2(X_b, \bR)$. By Proposition \ref{prop-HSKahler}, $[Q]^{1,1}$ is K\"{a}hler. So by Theorem \ref{thm-main1} (or more easily by Remark \ref{rem-11case}), we can construct a K\"{a}hler form on the total space $X$. 


\end{proof}

\begin{proof}[Proof of Theorem \ref{thm-HSfiber}]
Assume that $Q$ is a Hermitian-Symplectic form on $X$. 

In case 1, we know that the Kodaira-Spencer class is 0. So by the formula \eqref{eq-KS} we know that $\theta=0=\bar{\theta}$ in the decomposition of Gauss-Manin connection in \eqref{eq-Ddecomp}. As a consequence, the Gauss-Manin connection preserves the decomposition $\cE^0(B, V^2)=\cE^0(B, V^2)^{2,0}\oplus \cE^0(B, V^2)^{1,1}\oplus \cE^0(B, V^2)^{0,2}$. So the $(1,1)$-component of the flat section $[Q]$ is also flat. By Proposition \ref{prop-LZ}, we know that $[Q]^{1,1}$ restricts to K\"{a}hler class on each fiber. So we conclude by Proposition \ref{prop-HSKahler}. 

In case 2, there is a natural polarization of $V^1_{\bC}=R^1\pi_*\bR_{\bC}=H^{1,0}\oplus H^{0,1}$. Indeed, for any $b\in X_b$, $\xi, \eta\in H^{1,0}(X_b)$, define:
\begin{equation*}
\langle \xi, \bar{\eta}\rangle=\sqrt{-1}\int_{X_b} \xi\wedge \bar{\eta}\wedge Q^{n-1}. 
\end{equation*}
Then because $Q$ is closed and $Q^{1,1}$ is positive definite, it is easy to see that defines a polarization on $V^1_{\bC}$, which also induces a polarization on $V^2=\wedge^2 V^1$. So we can conclude by Proposition \ref{prop-HSKahler}. 

In both case 3 and case 4, we know that the $(1,1)$-component of $[Q]|_{X_b}$ satisfies the condition 2 of Proposition \ref{prop-HSKahler}. We also remark that the dimension of $H^{2,0}(X_b)$ does not depend on $b\in B$ by the invariance of Hodge numbers of K\"{a}hler manifolds under deformations (see \cite[Proposition 9.20, vI.9.3.2]{Voi02}). 

\end{proof}

We end this paper by briefly dicussing a natural problem, which is based on communications with Weiyi Zhang. 
\vskip 1mm

\noindent
\textbf{Problem:} Extend the main results in this paper to more general holomorphic maps. In fact, in the case of elliptic fibrations (with singular fibers) in high dimensions, Nakayama had derived similar type of K\"{a}hlerian criterion based on a delicate study of such structures (see \cite{Nak99a, Nak99}). On the other hand, Thurston's construction of symplectic structures has been generalized to the setting of Lefschetz pencils by Gompf (\cite{Gom05}). Correspondingly, we expect that Zucker's generalization of Deligne's result in \cite{Zuc79} and a careful treatment near the singular fibers of Lefschetz pencils should lead to a generalization of the K\"{a}hlerian criteria in this paper.


\vskip 3mm
\noindent
Department of Mathematics, Rutgers University, Piscataway, NJ 08854-8019.

\noindent
{\it E-mail address:} chi.li@rutgers.edu


\begin{thebibliography}{999999}

\bibitem
{AT14}
D. Angella, A. Tomassini,
Symplectic manifolds and cohomological decomposition, J. Symplectic Geometry, \textbf{12}, no. 2 (2014) 215-236. 

\bibitem
{Bla56}
A. Blanchard, Sur les vari\'{e}t\'{e}s analytiques complexes, Ann. Sci. \'{E}cole Norm. Sup. \textbf{73} (1956), 157-202.

\bibitem
{Chi14}
I. Chiose, Obstructions to the existence of K\"{a}hler structures on compact complex manifolds, Proc. Amer. Math. Soc. \textbf{142} (2014), no. 10, 3561–3568.


\bibitem
{Del68}
P. Deligne, Th\'{e}or\'{e}me de Lefschetz et crit\'{e}res de d\'{e}g\'{e}n\'{e}rescence de suite spectrales, Publ. Math. I.H.E.S., Vol. \textbf{35} (1968), pp. 107-126.

\bibitem
{DGMS}
P. Deligne, P. Griffiths, J. Morgan, and D. Sullivan, Real homotopy theory of K\"{a}hler manifolds, Invent. Math. \textbf{29} (1975), 245-275.


\bibitem
{DP04}
J.-P. Demailly, M. Paun, Numerical characterization of the Kähler cone of a compact Kähler manifold, Ann. of
Math. (2) \textbf{159} (3) (2004), 1247–1274.


\bibitem
{DP20}
S. Dinew, D. Popovici, 
A Generalised Volume Invariant for Aeppli Cohomology Classes of Hermitian-Symplectic Metrics, arXiv:2007.10647, to appear in Advances in Mathematics.  


\bibitem
{EFV12}
N. Enrietti, A. Fino, L. Vezzoni, Tamed Symplectic Forms and Strong K\"{a}hler with
Torsion Metrics, J. Symplectic Geom. \textbf{10}, No. 2 (2012) 203-223.

\bibitem
{FG65}
W. Fischer, H. Grauert,  Lokal-triviale Familien kompakter komplexer Mannigfaltigkeiten.
In: Nachr. Akad. Wiss. G\"{o}ttingen Math.-Phys. Kl. II 1965 \textbf{6} (1965), 89–94.


\bibitem
{FOT08}
 Y. F\'{e}lix, J. Oprea, and D. Tanr\'{e}, Algebraic models in geometry, Oxford Graduate Texts in Mathematics, vol. 17,
Oxford University Press, Oxford, 2008.

\bibitem
{GH78}
P. Griffiths, J. Harris, Principles of Algebraic Geometry. Wiley-Interscience (1978).

\bibitem
{Gom05}
R. Gompf, Locally holomorphic maps yield symplectic structures, Comm. Anal.
Geom. \textbf{13} (2005), no. 3, 511–525.

\bibitem
{HL83}
R. Harvey and H. B. Lawson, An intrinsic characterization of K\"{a}hler manifolds, Invent. Math. \textbf{74} (1983) 169-98


\bibitem
{LZ09}
T.-J. Li and W. Zhang, Comparing tamed and compatible symplectic cones and cohomological properties of almost complex manifolds, Comm. Anal. Geom. \textbf{17} (4) (2009), 651–683.

\bibitem
{Nak99a}
N. Nakayama, Projective algebraic varieties whose universal covering spaces are
biholomorphic to $\mathbb{C}^n$. J. Math. Soc. Japan, \textbf{51}(3):643–654, 1999.

\bibitem
{Nak99}
N. Nakayama, Compact K\"{a}hler manifolds whose universal covering spaces are biholomorphic to $\mathbb{C}^n$. 
Preprint available at http://www.kurims.kyoto-u.ac.jp/preprint/index.html No. 1230 (1999)

\bibitem
{Pet86}
T. Peternell, Algebraicity criteria for compact complex manifolds, Math. Ann. \textbf{275} (1986), no. 4, 653–672

\bibitem
{ST10}
J. Streets, G. Tian, A Parabolic Flow of Pluriclosed Metrics, Int. Math. Res. Notices, \textbf{16} (2010), 3101-3133.

\bibitem
{Thu76}
W. Thurston, Some simple examples of symplectic manifolds, Proc. Amer.
Math. Soc. \textbf{55} (1976), 467-468.

\bibitem
{Ver14}
M. Verbitsky, Rational Curves and Special Metrics on Twistor Spaces, Geometry and
Topology \textbf{18} (2014), 897–909.

\bibitem
{Voi02}
C. Voisin, Hodge theory and complex algebraic geometry, Volume I and II. Cambridge Stud. Adv. Math. \textbf{76}. (2002)

\bibitem
{Ye22}
Y. Ye, Pluriclosed flow and Hermitian-symplectic structures, arXiv:2207.12643



\bibitem
{Zuc79}
S. Zucker, Hodge theory with degenerating coefficients: $L^2$-cohomology in the Poincar\'{e} metric, Ann of Math. (2) \textbf{109} (1979), 415-476. 


\end{thebibliography}
\end{document}